\documentclass[11pt,reqno]{amsart}

\usepackage[T1]{fontenc}
\usepackage{lmodern}
\usepackage{microtype}
\usepackage[margin=1.03in]{geometry}
\usepackage{amsmath,amssymb,amsfonts,mathtools}
\usepackage{amsthm}
\usepackage{booktabs}
\usepackage{enumitem}
\usepackage{xcolor}
\usepackage[colorlinks=true,
  linkcolor=blue!50!black,
  citecolor=blue!50!black,
  urlcolor=blue!50!black]{hyperref}
\usepackage[capitalize,noabbrev]{cleveref}

\allowdisplaybreaks
\numberwithin{equation}{section}
\setlist[enumerate]{leftmargin=2.4em,itemsep=0.25em,topsep=0.4em}
\setlist[itemize]{leftmargin=2em,itemsep=0.25em,topsep=0.4em}

\theoremstyle{plain}
\newtheorem{theorem}{Theorem}[section]
\newtheorem{proposition}[theorem]{Proposition}
\newtheorem{lemma}[theorem]{Lemma}

\theoremstyle{definition}
\newtheorem{definition}[theorem]{Definition}
\newtheorem{question}[theorem]{Question}
\crefname{question}{Question}{Questions}
\Crefname{question}{Question}{Questions}

\theoremstyle{remark}

\newcommand{\Ann}{\operatorname{Ann}}
\newcommand{\Supp}{\operatorname{supp}}
\newcommand{\A}{\mathbb A}
\newcommand{\N}{\mathbb N}
\newcommand{\kk}{k}
\newcommand{\m}{\mathfrak m}

\title[A Reduced Non-McCoy Ring]
{A Reduced Non-McCoy Ring with\\
the Annihilator Condition and \((A_2)\)}

\author{Viet-Hoang Tran}
\address{Department of Mathematics, National University of Singapore, Singapore 119076}
\email{hoang.tranviet@u.nus.edu}
\urladdr{https://vh-tran.github.io/}

\author{Tan M. Nguyen}
\address{Department of Mathematics, National University of Singapore, Singapore 119076}
\email{tanmn@nus.edu.sg}
\urladdr{https://tanmnguyen89.github.io/}

\date{}


\hypersetup{
  pdftitle={A Reduced Non-McCoy Ring with the Annihilator Condition and (A2)},
  pdfauthor={Viet-Hoang Tran}
}

\begin{document}
\raggedbottom

\begin{abstract}
It has been asked whether a reduced non-McCoy ring can satisfy both
the annihilator condition and \((A_n)\) for some \(n\geq2\).
In this work, we prove that such rings exist exactly when \(n=2\),
by constructing a countable example and showing that the annihilator
condition together with \((A_3)\) forces a ring to be McCoy.
\end{abstract}

\maketitle

\section{Introduction}

All rings in this paper are commutative, nonzero, and have an identity.
The set of zero-divisors of a ring \(R\), including \(0\), is denoted by
\(Z(R)\); an element outside \(Z(R)\) is called \emph{regular}.  For a
subset \(E\subseteq R\), put
\[
  \Ann_R(E)=\{r\in R:rE=0\}.
\]
We abbreviate \(\Ann_R(\{a_1,\ldots,a_s\})\) to
\(\Ann_R(a_1,\ldots,a_s)\).
A ring is \emph{reduced} if it has no nonzero nilpotent elements.

\begin{definition}
Let \(R\) be a ring.
\begin{enumerate}[label=\textup{(\roman*)}]
\item \(R\) has the \emph{annihilator condition}, abbreviated
\emph{(a.c.)}, if for every \(a,b\in R\) there is \(c\in R\) such that
\[
  \Ann_R(a,b)=\Ann_R(c).
\]
\item For \(n\geq 1\), \(R\) has \emph{\((A_n)\)} if every ideal generated
by at most \(n\) elements and contained in \(Z(R)\) has a nonzero
annihilator.
\item \(R\) is a \emph{McCoy ring} if every finitely generated ideal
contained in \(Z(R)\) has a nonzero annihilator.
\end{enumerate}
\end{definition}

The condition \((A_1)\) is automatic: if \((a)\subseteq Z(R)\), then
\(a\) is a zero-divisor and
\(\Ann_R((a))=\Ann_R(a)\neq0\); for \(a=0\), this annihilator is \(R\).
Since every ideal generated by at most \(n\) elements is also generated
by at most \(n+1\) elements, \((A_{n+1})\) implies \((A_n)\).
By definition, \(R\) is McCoy if and only if it satisfies \((A_n)\)
for every \(n\geq1\).

McCoy proved in 1942 that a polynomial \(\sum_i a_iT^i\in R[T]\) is a zero-divisor
exactly when a nonzero element of \(R\) annihilates all its coefficients
\cite{McCoy1942}.  The annihilator condition appeared in the work of
Henriksen and Jerison on minimal prime spectra
\cite{HenriksenJerison1962,HenriksenJerison1965}.  The McCoy condition
was studied as Condition~(C) by Quentel and as Property~(A) by Huckaba
and Keller \cite{Quentel1971,Quentel1972Erratum,HuckabaKeller1979};
we use Faith's terminology \cite{Faith1991}.

Lucas showed that (a.c.) and the McCoy condition are incomparable even among
reduced rings, and that \((A_n)\) does not imply \((A_{n+1})\) for any
\(n\geq2\), again within the reduced class
\cite[Examples~2.2, 2.4, and 2.5]{Lucas1986}.  These examples lead to
the simultaneous existence question:
\begin{question}\label{question:main}
Does there exist a reduced ring \(R\) satisfying \emph{(a.c.)} and
\((A_n)\) for some integer \(n\geq2\) that is not McCoy?
\end{question}
This is listed as Problem~10 in the open-problem collection of
Cahen, Fontana, Frisch, and Glaz, and is credited there to T.~G.~Lucas
\cite[Problem~10, p.~359]{CahenFontanaFrischGlaz2014}.  Subsequent work
has studied the McCoy condition, annihilator conditions, their behavior under
extensions, and stronger finite zero-divisor conditions
\cite{HongKimLeeRyu2007,HongKimLeeNielsen2009,
BajorZiembowski2019,MalkiLouzari2024,KimZiembowski2025}.
Our main result answers the question by construction.

\begin{theorem}\label{thm:main}
There exists a countable reduced ring satisfying \emph{(a.c.)} and
\((A_2)\) but not \((A_3)\).  In particular, this ring is not McCoy.
\end{theorem}

The value \(n=2\) is sharp: \cref{prop:threshold} shows that (a.c.)
and \((A_3)\) imply that a ring is McCoy, without assuming reducedness.
In \cref{tab:threshold}, the first row is definitional; the remaining
rows are the construction and threshold results proved in this paper.

\begin{table}[htbp]
\centering
\caption{The conditions \((A_n)\) and McCoy rings.}
\label{tab:threshold}
\footnotesize
\renewcommand{\arraystretch}{1.15}
\begin{tabular}{@{}l@{\hspace{1em}}l@{}}
\toprule
Conditions on \(R\) & Conclusion\\
\midrule
\((A_n)\) for every \(n\geq1\)
& \(R\) is McCoy, and conversely (definition).\\
\addlinespace[2pt]
Reduced, (a.c.), and \((A_2)\)
& A countable non-McCoy example exists (\cref{thm:main}).\\
\addlinespace[2pt]
(a.c.) and \((A_n)\) for fixed \(n\geq3\)
& \(R\) is necessarily McCoy (\cref{prop:threshold}).\\
\bottomrule
\end{tabular}
\end{table}

Lucas constructed a reduced ring satisfying \((A_3)\) but not
\((A_4)\) \cite[Example~2.5]{Lucas1986}.  Since every McCoy ring
satisfies \((A_4)\), this example shows that (a.c.) cannot be omitted
from \cref{prop:threshold}.  Reducedness is not used in that
proposition, but is required in \cref{question:main}.  The ring in
\cref{thm:main} is reduced because it is a subring of a product of fields.

After proving \cref{prop:threshold}, we translate annihilators into
zero sets of functions and construct a suitable countable set in
punctured affine \(3\)-space.  Infinite repetition of its points and
adjunction of finite-support functions then give the required ring.

\section{The three-generator threshold}\label{sec:threshold}

\begin{proposition}\label{prop:threshold}
Every ring satisfying \emph{(a.c.)} and \((A_3)\) is McCoy.
Consequently, for each fixed \(n\geq3\), a ring with \emph{(a.c.)}
satisfies \((A_n)\) if and only if it is McCoy.
\end{proposition}

\begin{proof}
Let \(R\) satisfy (a.c.) and \((A_3)\).  Suppose that \(R\) is not
McCoy.  Among all finite lists generating an ideal contained in
\(Z(R)\) with zero annihilator, choose one of least length, say
\[
  I=(a_1,\ldots,a_m)\subseteq Z(R),
  \qquad \Ann_R(I)=0.
\]
Such a list exists because \(R\) is not McCoy, and \((A_3)\)
forces \(m\geq4\).  By (a.c.), choose \(t\in R\) with
\(\Ann_R(t)=\Ann_R(a_1,a_2)\), and set
\[
  J=(t,a_3,\ldots,a_m).
\]
Then
\[
  \Ann_R(J)
  =\Ann_R(t)\cap\bigcap_{i=3}^{m}\Ann_R(a_i)
  =\Ann_R(I)=0.
\]

We show that \(J\subseteq Z(R)\).  Given \(v\in J\), write
\(v=rt+q\), where \(r\in R\) and \(q\in(a_3,\ldots,a_m)\).
The three-generated ideal
\[
  K=(a_1,a_2,q)\subseteq I\subseteq Z(R)
\]
has nonzero annihilator by \((A_3)\).  Choose
\(0\neq u\in\Ann_R(K)\).  Then
\(ua_1=ua_2=uq=0\); the choice of \(t\) gives \(ut=0\), and
\[
  uv=r(ut)+uq=0.
\]
Hence \(v\in Z(R)\).  The element \(u\) may depend on \(v\), which
is sufficient to prove \(J\subseteq Z(R)\).  Since \(J\) has zero
annihilator and at most \(m-1\) generators, this contradicts the
minimality of \(m\).  Therefore \(R\) is McCoy.

For any fixed \(n\geq3\), \((A_n)\) implies \((A_3)\), so the same
conclusion holds.  The reverse implication follows directly from the
definition of a McCoy ring.
\end{proof}

The role of (a.c.) is to replace two generators by one while preserving
their common annihilator.  The additional use of \((A_3)\) proves that
the replacement ideal still consists entirely of zero-divisors;
neither \(t\in I\) nor \(J\subseteq I\) is assumed.  This step tests the
three-generated ideals \((a_1,a_2,q)\), so \((A_2)\) does not supply
the required conclusion.  The remaining sections construct the
reduced example showing that this distinction is sharp.

\section{Function rings and annihilators}\label{sec:functions}

Let \(X\) be a nonempty set, let \(\kk\) be a field, and let
\(B\subseteq\kk^X\) be a unital \(\kk\)-subalgebra.  Put
\(\Omega=X\times\N\), where \(\N=\{1,2,\ldots\}\).  For \(f\in B\),
write \(\Delta f\) for the function on \(\Omega\) with
\((\Delta f)(p,j)=f(p)\).  Define
\begin{equation}\label{eq:trace-ring}
  D=\kk^{(\Omega)}
   =\{d\in\kk^\Omega:\Supp(d)\text{ is finite}\},
  \qquad R=\Delta B+D\subseteq\kk^\Omega,
\end{equation}
where \(\Supp(d)=\{\omega:d(\omega)\neq0\}\).  Since \(D\) is an ideal
of \(\kk^\Omega\), \(R\) is a unital subring.  For a function \(f\) on
a set \(Y\), write \(Z_Y(f)=\{p\in Y:f(p)=0\}\).

\begin{lemma}\label{prop:dictionary}
The ring \(R\) in \eqref{eq:trace-ring} has the following properties.
\begin{enumerate}[label=\textup{(\alph*)}]
\item \(R\) is reduced, and for every finite tuple \(r_1,\ldots,r_s\in R\),
\begin{equation}\label{eq:ann-dictionary}
 \Ann_R(r_1,\ldots,r_s)
 =
 \{c\in R:\Supp(c)\subseteq\textstyle\bigcap_{i=1}^{s}Z_\Omega(r_i)\}.
\end{equation}
This annihilator is nonzero exactly when the tuple has a common
zero in \(\Omega\).  Equal common zero sets give equal annihilators,
and an element of \(R\) is regular exactly when it is zero-free on
\(\Omega\).
\item \(R\) satisfies \emph{(a.c.)} if for every \(f,g\in B\) there is
\(t\in B\) such that
\begin{equation}\label{eq:binary-coder}
  Z_X(t)=Z_X(f)\cap Z_X(g).
\end{equation}
\item \(R\) satisfies \((A_2)\) if every \(f,g\in B\) with no common
zero on \(X\) admit \(u,v\in B\) such that \(uf+vg\) is zero-free on \(X\).
\end{enumerate}
\end{lemma}

\begin{proof}
A subring of a product of fields is reduced.  Coordinatewise
multiplication gives \eqref{eq:ann-dictionary}.  If a tuple has a common
zero \(\omega\), its annihilator contains the nonzero idempotent
\(\delta_\omega\in D\) that is \(1\) at \(\omega\) and \(0\) elsewhere.
If the common zero set is empty, \eqref{eq:ann-dictionary} makes the
annihilator zero.  The remaining assertions of (a) follow from the
same formula.

For (b), write \(r=\Delta f+d_1\) and \(s=\Delta g+d_2\), where
\(d_1,d_2\in D\), and choose \(t\) as in \eqref{eq:binary-coder}.
Let \(F=\Supp(d_1)\cup\Supp(d_2)\).  Outside \(F\), \(\Delta t\)
vanishes exactly where both \(r\) and \(s\) vanish.  Choose \(e\in D\),
supported on \(F\), so that
\[
 (\Delta t+e)(\omega)=
 \begin{cases}
  0,&r(\omega)=s(\omega)=0,\\
  1,&\text{otherwise}
 \end{cases}
 \qquad(\omega\in F).
\]
Then \(Z_\Omega(\Delta t+e)=Z_\Omega(r)\cap Z_\Omega(s)\), and (a)
gives \(\Ann_R(\Delta t+e)=\Ann_R(r,s)\).

For (c), we use the contrapositive formulation of \((A_2)\):
\begin{equation}\label{eq:A2-contrapositive}
 \Ann_R(r,s)=0
 \quad\Longrightarrow\quad
 (r,s)\text{ contains a regular element}.
\end{equation}
This is equivalent to the definition because an ideal is not contained
in \(Z(R)\) precisely when it contains a regular element; principal
ideals are covered by taking \(s=0\).

Suppose that \(r=\Delta f+d_1\) and \(s=\Delta g+d_2\) have zero common
annihilator.  By (a), they have no common zero in \(\Omega\).  The
functions \(f,g\) have no common zero on \(X\): such a zero \(p\)
would give a common zero of \(r,s\) at any coordinate of the infinite
fiber \(\{p\}\times\N\) outside
\(\Supp(d_1)\cup\Supp(d_2)\).  Choose \(u,v\in B\) so that
\(h=uf+vg\) is zero-free on \(X\).  Then
\[
  w_0=\Delta u\,r+\Delta v\,s=\Delta h+e
  \qquad\text{for some }e\in D.
\]
Its zero set \(E\) is contained in the finite set \(\Supp(e)\).
At each \(\omega\in E\), at least one of \(r(\omega),s(\omega)\) is
nonzero.  Define \(c_1,c_2\in D\), supported on \(E\), by
\[
 (c_1(\omega),c_2(\omega))=
 \begin{cases}
  (r(\omega)^{-1},0),&r(\omega)\neq0,\\
  (0,s(\omega)^{-1}),&r(\omega)=0.
 \end{cases}
\]
The element \(w=w_0+c_1r+c_2s\in(r,s)\) is \(1\) on \(E\) and
agrees with \(w_0\) elsewhere.  It is therefore zero-free, hence regular
by (a), proving \eqref{eq:A2-contrapositive}.
\end{proof}

The infinite fibers ensure that common zeros survive finitely supported
perturbations, while the functions in \(D\) allow corrections at finitely
many coordinates.  These features will also yield the three-generated
ideal in the final construction.

\section{A countable set in punctured affine space}

For the rest of the paper, fix the countable algebraically closed field
\[
  \kk=\overline{\mathbb Q},
  \qquad
  A=\kk[x,y,z],
  \qquad
  \m=(x,y,z),
\]
and let \(o=(0,0,0)\in\A^3_\kk\).  By a \emph{curve} we mean an
irreducible closed subvariety of \(\A^3_\kk\) of dimension one.

\begin{lemma}\label{lem:elementary-geometry}
The following facts hold.
\begin{enumerate}[label=\textup{(\roman*)}]
\item The sets of polynomials in \(A\), ordered pairs in \(A^2\), and
curves through \(o\) are countable.
\item A proper closed subset of an irreducible affine curve is finite,
and an affine curve over \(\kk\) has infinitely many \(\kk\)-points.
\item For every \(0\neq a\in A\), infinitely many lines through \(o\)
are not contained in \(V(a)\).
\item If \(f,g\in A\) are nonzero and have greatest common divisor
\(d\), write \(f=df_0\), \(g=dg_0\).  Then
\[
 W(f,g)=V(f_0,g_0)
 \tag{4.1}\label{eq:residual-locus}
\]
is either empty or a finite union of irreducible affine curves.
\end{enumerate}
\end{lemma}

\begin{proof}
The field \(\overline{\mathbb Q}\) is countable, hence so are \(A\) and
\(A^2\).  Since \(A\) is Noetherian, every ideal is finitely generated;
there are consequently only countably many ideals and only countably
many curves through \(o\).  This proves (i).

A proper closed subset of an irreducible one-dimensional Noetherian
space has dimension zero and therefore consists of finitely many
closed points.  Noether normalization gives a finite dominant morphism
from any affine curve to \(\A^1_\kk\).  Its image is both dense and
closed, hence it is surjective.  Every fiber over a \(\kk\)-point
contains a closed point, whose residue field is \(\kk\) because \(\kk\)
is algebraically closed.  As \(\kk\) is infinite, the curve has
infinitely many \(\kk\)-points.  This proves (ii).

Write \(a=\sum_{r=0}^N a_r\), where the \(a_r\) are homogeneous, and
choose \(r\) with \(a_r\neq0\).  Every direction
\([v]\in\mathbb P^2(\kk)\) satisfying \(a_r(v)\neq0\) determines a
line on which \(a\) is not identically zero.  These directions form a
nonempty Zariski-open subset of \(\mathbb P^2_\kk\), hence an infinite
set.  This proves (iii).

For (iv), \(A\) is a unique factorization domain and
\(\gcd(f_0,g_0)=1\).  Let \(\mathfrak p\) be a prime minimal over
\((f_0,g_0)\).  Krull's height theorem gives
\(\operatorname{ht}\mathfrak p\leq2\).  Height zero is impossible when
the common zero set is nonempty, because \(A\) is a domain and
\((f_0,g_0)\neq0\).  A height-one prime in the UFD \(A\) is generated
by an irreducible element; if it contained both \(f_0\) and \(g_0\),
that element would be a common divisor.  Thus every minimal prime has
height two.  The irreducible components of \(W(f,g)\) consequently
have dimension \(3-2=1\), and there are finitely many of them.
\end{proof}

\begin{proposition}\label{prop:thin}
There is a countable set
\[
  X=\{p_1,p_2,\ldots\}\subseteq\A^3_\kk\setminus\{o\}
\]
with the following properties.
\begin{enumerate}[label=\textnormal{(T\arabic*)},ref=T\arabic*]
\item\label{T:dense}
\(X\) is Zariski dense in \(\A^3_\kk\).
\item\label{T:thin}
Every curve through \(o\) meets \(X\) in a nonempty finite set.
\item\label{T:pair-hit}
If \(f,g\in\m\), then \(f\) and \(g\) have a common zero in \(X\).
\item\label{T:coder}
For every \(f,g\in A\), there is \(t\in A\) such that
\[
  Z_X(t)=Z_X(f)\cap Z_X(g).
  \tag{4.2}\label{eq:T-coder}
\]
\item\label{T:escape}
If \(f,g\in A\) have no common zero on \(X\), there are \(u,v\in A\)
such that \(uf+vg\) is zero-free on \(X\).
\end{enumerate}
\end{proposition}

\begin{proof}
Enumerate the curves through \(o\) without repetition as
\[
 C_1,C_2,\ldots,
\]
enumerate the nonzero polynomials as \(a_1,a_2,\ldots\), and enumerate
all ordered pairs of polynomials as
\[
  (f_1,g_1),(f_2,g_2),\ldots.
  \tag{4.3}\label{eq:enumerations}
\]
For every index \(j\) with \(f_jg_j\neq0\), form the residual locus
\(W_j=W(f_j,g_j)\) from \eqref{eq:residual-locus}.  By
\cref{lem:elementary-geometry}(iv), the components of \(W_j\) that
contain \(o\) occur among finitely many of the \(C_i\).  Let
\(\ell(j)\) be the largest of their indices, with \(\ell(j)=0\) when
there is no such component, and set
\[
  \nu(j)=\max\{j,\ell(j)\}.
  \tag{4.4}\label{eq:activation}
\]
When \(f_jg_j=0\), define \(\nu(j)=j\).
For a fixed \(n\), only finitely many indices satisfy \(\nu(j)=n\),
because \(\nu(j)\geq j\).

We recursively choose \(p_n\in C_n\setminus\{o\}\) and maintain a
finite family \(\mathcal F_n\) of permanently forbidden closed sets.
The induction invariant is
\[
 \text{no member of \(\mathcal F_n\) contains \(C_r\) for any \(r>n\)}.
 \tag{4.5}\label{eq:invariant}
\]
Start with \(\mathcal F_0=\varnothing\).  Assume all choices through
stage \(n-1\) have been made.  On \(C_n\), avoid \(o\), every member
of \(\mathcal F_{n-1}\), and, for each \(j\leq n\) for which
\(C_n\nsubseteq V(a_j)\), avoid \(V(a_j)\).  By
\eqref{eq:invariant}, every member of \(\mathcal F_{n-1}\) meets
\(C_n\) properly; the same holds for each \(V(a_j)\) by the stated
condition \(C_n\nsubseteq V(a_j)\).  There are finitely many of these
sets, and each
intersection is finite by \cref{lem:elementary-geometry}(ii).
Since \(C_n(\kk)\) is infinite, a point \(p_n\) remains.  Add \(C_n\)
to the forbidden family.  Notice that the \(p_n\) are distinct:
after stage \(i\), every later point avoids \(C_i\), which contains
\(p_i\).

After choosing \(p_n\), we impose the following requirements.

\smallskip
\noindent\emph{Intersections of zero sets.}
For every \(j\) with \(\nu(j)=n\), do the following.  If
\(f_j=g_j=0\), set \(t_j=0\).  If exactly one is zero,
let \(t_j\) be the other polynomial.  Suppose now that both are
nonzero.  Write
\[
 f_j=d_jf_j',\qquad g_j=d_jg_j',
 \qquad \gcd(f_j',g_j')=1,
\]
and let \(W_j=V(f_j',g_j')\).  The maximal ideals of the distinct
points \(o,p_1,\ldots,p_n\) are pairwise comaximal.  The Chinese
remainder theorem therefore gives a polynomial \(q_j\in A\) such that
\[
  q_j(o)=1,\qquad
  q_j(p_i)=
  \begin{cases}
    0,&p_i\in W_j,\\
    1,&p_i\notin W_j
  \end{cases}
  \quad(1\leq i\leq n).
  \tag{4.6}\label{eq:q-interpolation}
\]
Set \(t_j=d_jq_j\), and add \(W_j\) and \(V(q_j)\) to the forbidden
family.

\smallskip
\noindent\emph{Linear-combination requirement.}
Consider the single pair \((f_n,g_n)\).  If \(f_n,g_n\in\m\), make no
additional choice.  If at least one is outside \(\m\), but the two
polynomials already have a common zero among \(p_1,\ldots,p_n\), again
make no additional choice.  Otherwise, at every point of
\(\{o,p_1,\ldots,p_n\}\), at least one of \(f_n,g_n\) is nonzero.
Set \(p_0=o\), and choose scalars \(\alpha_i,\beta_i\in\kk\) such that
\[
 \alpha_i f_n(p_i)+\beta_i g_n(p_i)=1.
\]
Apply the Chinese remainder theorem separately to the prescribed
values \(\alpha_i\) and \(\beta_i\).  This produces \(u_n,v_n\in A\)
for which
\[
 h_n=u_nf_n+v_ng_n
 \quad\text{satisfies}\quad
 h_n(o)=h_n(p_1)=\cdots=h_n(p_n)=1.
 \tag{4.7}\label{eq:h-interpolation}
\]
Add \(V(h_n)\) to the forbidden family.

This completes stage \(n\).  We verify the invariant.  The newly
forbidden curve \(C_n\) contains no different curve through \(o\).  The sets
\(V(q_j)\) and \(V(h_n)\) contain no curve through \(o\) because their
defining polynomials take the value \(1\) at \(o\).  If \(W_j\) is
added at stage \(n=\nu(j)\), every curve through \(o\) contained in \(W_j\)
is an irreducible component of \(W_j\), by
\cref{lem:elementary-geometry}(iv), and its index is at most
\(\ell(j)\leq n\).  Thus \(W_j\) contains no curve \(C_r\) with \(r>n\).
This proves \eqref{eq:invariant} and hence completes the recursion.

We now establish the five asserted properties.  Fix \(0\neq a_j\in A\).
By \cref{lem:elementary-geometry}(iii), infinitely many lines through \(o\)
are not contained in \(V(a_j)\).  Their indices in the enumeration of
the \(C_i\) are unbounded.  For every such index \(n\geq j\), the
choice rule gives \(a_j(p_n)\neq0\).  This proves property
\ref{T:dense}.  If \(C=C_j\), then \(p_j\in C\), while every
\(p_i\) with \(i>j\) avoids the permanently forbidden set \(C_j\).
Thus
\[
  \varnothing\neq X\cap C_j\subseteq\{p_1,\ldots,p_j\},
\]
which proves property \ref{T:thin}.

For property \ref{T:pair-hit}, let \(f,g\in\m\).  We first show that
some curve through \(o\) is contained in \(V(f,g)\).  Work in the
three-dimensional local domain \(A_\m\), and let \(P\) be a prime
minimal over \((f,g)A_\m\).  Krull's height theorem gives
\(\operatorname{ht}P\leq2\).  Contract \(P\) to a prime
\(\mathfrak p\subseteq\m\) of \(A\); localization preserves its height.
If \(\operatorname{ht}\mathfrak p=2\), then
\(V(\mathfrak p)\) is the required curve through \(o\).  If
\(\operatorname{ht}\mathfrak p=1\), the UFD property gives
\(\mathfrak p=(\pi)\) for an irreducible \(\pi\).  Choose
\(s\in\m\setminus\mathfrak p\), and in \(A_\m\) choose a prime minimal
over \((\pi,s)A_\m\).  Its contraction \(\mathfrak q\) is contained in
\(\m\) and strictly contains \(\mathfrak p\), so
\(\operatorname{ht}\mathfrak q\geq2\).  Krull's height theorem applied
to the two generators \(\pi,s\) gives
\(\operatorname{ht}\mathfrak q\leq2\).  Hence
\(\operatorname{ht}\mathfrak q=2\), and
\(V(\mathfrak q)\subseteq V(f,g)\) is a curve through \(o\).  Finally, if
\(\operatorname{ht}\mathfrak p=0\), then
\(\mathfrak p=0\); the injectivity \(A\hookrightarrow A_\m\) gives
\(f=g=0\), and \(V(x,y)\) is a curve through \(o\) in \(V(f,g)\).
By property \ref{T:thin}, the resulting curve through \(o\) meets \(X\),
giving a common zero of \(f,g\).  This proves property
\ref{T:pair-hit}.

To prove property \ref{T:coder}, consider the pair \((f_j,g_j)\),
processed at stage \(n=\nu(j)\).  The cases in which a polynomial
is zero are immediate.  In the remaining case,
\[
 Z_X(f_j)\cap Z_X(g_j)
 =Z_X(d_j)\cup\bigl(X\cap W_j\bigr).
 \tag{4.8}\label{eq:gcd-zero}
\]
Equation \eqref{eq:q-interpolation} gives
\(q_j(p_i)=0\) exactly when \(p_i\in W_j\) for \(i\leq n\).
Every later point avoids both \(W_j\) and \(V(q_j)\).  Consequently
\[
  Z_X(q_j)=X\cap W_j.
\]
Since \(t_j=d_jq_j\), \eqref{eq:gcd-zero} gives
\(Z_X(t_j)=Z_X(f_j)\cap Z_X(g_j)\), as required.

Finally, suppose \(f_j,g_j\) have no common zero on \(X\).
Property \ref{T:pair-hit} shows that they cannot both belong to \(\m\).
They had no common zero among the points present at stage \(j\), so
the recursion constructed \(h_j=u_jf_j+v_jg_j\) as in
\eqref{eq:h-interpolation}.  It is nonzero at the points already
chosen, and every later point avoids the permanently forbidden set
\(V(h_j)\).  Thus \(h_j\) is zero-free on \(X\), proving property
\ref{T:escape}.
\end{proof}

\section{The reduced example}\label{sec:example}

Let \(X\) be the set of \cref{prop:thin}.  By property \ref{T:dense},
restriction gives an injective homomorphism
\[
  \rho:A\longrightarrow\kk^X,\qquad f\longmapsto f|_X.
\]
Put \(B=\rho(A)\), \(\Omega=X\times\N\), \(D=\kk^{(\Omega)}\), and
\(\iota=\Delta\circ\rho\).  We take
\begin{equation}\label{eq:final-ring}
  R=\iota(A)+D\subseteq\kk^\Omega.
\end{equation}

\begin{proof}[Proof of \cref{thm:main}]
The ring \(R\) is reduced by \cref{prop:dictionary}(a).
Property \ref{T:coder} and \cref{prop:dictionary}(b) give (a.c.);
property \ref{T:escape} and \cref{prop:dictionary}(c) give \((A_2)\).

Consider the ideal
\[
  I=(\iota(x),\iota(y),\iota(z))R.
\]
Its generators have no common zero in \(\Omega\), since the only common
zero of \(x,y,z\) in \(\A^3_\kk\) is \(o\), which does not belong to
\(X\).  Thus \cref{prop:dictionary}(a) gives \(\Ann_R(I)=0\).

For an arbitrary \(q\in I\), write
\[
  q=\iota(x)r_1+\iota(y)r_2+\iota(z)r_3,
  \qquad r_i=\iota(f_i)+d_i,\quad f_i\in A,\ d_i\in D.
\]
Because \(D\) is an ideal, this has the form \(q=\iota(F)+d\), where
\[
  F=xf_1+yf_2+zf_3\in\m,\qquad d\in D.
\]
Property \ref{T:pair-hit}, applied to \((F,0)\), gives \(p\in X\)
with \(F(p)=0\).  Some coordinate of the infinite fiber
\(\{p\}\times\N\) lies outside the finite set \(\Supp(d)\), and \(q\)
vanishes there.  By \cref{prop:dictionary}(a), \(q\) is a zero-divisor.
We have proved
\begin{equation}\label{eq:bad-ideal}
  I\subseteq Z(R),\qquad \Ann_R(I)=0.
\end{equation}
This three-generated ideal witnesses the failure of \((A_n)\) for
every \(n\geq3\).  Since \((A_1)\) is automatic and \((A_2)\) holds,
\(R\) satisfies \((A_n)\) exactly for \(n=1,2\), and is not McCoy.

Finally, \(\kk=\overline{\mathbb Q}\), \(A\), and \(\Omega\) are
countable.  There are countably many finite subsets of \(\Omega\),
and countably many \(\kk\)-valued functions supported on each.
Hence \(D\), and therefore \(R=\iota(A)+D\), is countable.
\end{proof}

\renewcommand{\sectionname}{}
\section*{Acknowledgments}
We used GPT-5.6 Sol and an agentic harness built around GPT-5.6 Sol and Claude Fable 5 to assist with literature searches, hypothesis testing, the exploration and elimination of potential approaches, wording refinement, and manuscript proofreading. We thank Hieu M. Vu, Tho Tran Huu, Khoi M. N. Nguyen, Dung V. Nguyen, and Quang X. Nguyen for their assistance with hardware-related matters and for providing technical support in the use of the AI tools and agentic system. Viet-Hoang Tran thanks Thieu N. Vo for insightful discussions regarding the problem.

\bibliographystyle{plain}
\begingroup
\footnotesize
\raggedright
\bibliography{references}
\endgroup

\end{document}